\documentclass{amsart}

\RequirePackage{fix-cm}
\usepackage{graphicx}
\usepackage{amsmath, amsopn,amstext,amscd,amsfonts,amssymb}
\usepackage{dsfont}
\usepackage{comment}
\usepackage[active]{srcltx}
\usepackage{graphicx, epsfig, subfig}

\usepackage{textcomp}

\usepackage{algorithm}

\makeatletter
\def\BState{\State\hskip-\ALG@thistlm}
\makeatother

\def\downbar#1{
\setbox10=\hbox{$#1$}
            \dimen10=\ht10 \advance\dimen10 by 2.5pt
            \ifdim \dimen10<15pt 
               \advance\dimen10 by -0.5pt
               \dimen11=\dimen10
               \advance\dimen10 by 2.5pt
               \lower \dimen11
            \else \lower \ht10 \fi
            \hbox {\hskip 1.5pt \vrule height \dimen10 depth \dp10}}
\def\upbar#1{
\setbox10=\hbox{$#1$}
            \dimen10=\ht10 \advance\dimen10 by \dp10 \advance\dimen10 by 2.5pt
            \ifdim \dimen10<15pt 
                \advance\dimen10 by 2pt \fi
            \raise 2.5pt \hbox {\hskip -1.5pt \vrule height \dimen10}}

\usepackage{multicol}
\usepackage{colortbl}
\usepackage{exscale}
\usepackage{amssymb,latexsym,amsthm,amsfonts,color,fancyhdr,mathrsfs}
\usepackage{lscape}
\usepackage{rotating}
\usepackage{caption}

\newtheorem{theorem}{\bf Theorem}[section]

\newtheorem{corollary}{\bf Corollary}[section]
\newtheorem{remark}{\bf Remark}[section]

\numberwithin{equation}{section}
\begin{document}
\title[Classical orthogonal polynomials]{On another characterization of Askey-Wilson polynomials}

\author{D. Mbouna}
\address{D. Mbouna \\University of Almeria, Dep. Mathematics, Almeria, Spain}
\email{mbouna@ual.es}
\address{A. Suzuki \\University of Coimbra, CMUC, Dep. Mathematics, 3001-501 Coimbra, Portugal}
\author{A. Suzuki}
\email{uc46263@uc.pt}

\date{\today}

\begin{abstract}
In this paper we show that the only sequences of orthogonal polynomials $(P_n)_{n\geq 0}$ satisfying
\begin{align*}
\phi(x)\mathcal{D}_q P_{n}(x)=a_n\mathcal{S}_q P_{n+1}(x) +b_n\mathcal{S}_q P_n(x) +c_n\mathcal{S}_q P_{n-1}(x),
\end{align*}
($c_n\neq 0$) where $\phi$ is a well chosen polynomial of degree at most two, $\mathcal{D}_q$ is the Askey-Wilson operator and $\mathcal{S}_q$ the averaging operator, are the multiple of Askey-Wilson polynomials, or specific or limiting cases of them.
\end{abstract}
\maketitle

\section{Introduction}\label{introduction}
In 1972, Al-Salam and Chihara proved (see \cite{Al-Salam-1972}) that $(P_n)_{n\geq 0}$ is a $\mathrm{D}$-classical orthogonal polynomial sequence (OPS), namely Hermite, Laguerre, Bessel or Jacobi families, if and only if 
\begin{align}
(az^2+bz+c)\mathrm{D}P_n(z)=(a_nz+b_n)P_n(z)+c_nP_{n-1}(z)\quad (c_n\neq 0)\;, \label{very-classical}
\end{align}
where $\mathrm{D}=d/dz$. We can replace $\mathrm{D}$ in \eqref{very-classical} by the following Askey-Wilson operator
\begin{align*}
\mathcal{D}_q\, p(x(s))= \frac{p(x(s+1/2))-p(x(s-1/2))}{x(s+1/2)-x(s-1/2)},\quad x(s)=\mbox{$\frac12$}(q^s +q^{-s})\;,
\end{align*}
for every polynomial $p$. We assume that $0<q<1$. (Taking $q^s=e^{i \theta}$ we recover $\mathcal{D}_q$ as defined in  \cite[(21.6.2)]{I2005}.) The problem of characterizing such OPS was posed by Ismail (see \cite[Conjecture 24.7.8]{I2005}). The case $a=b=0$ and $c=1$ was considered by Al-Salam (see \cite{A-1995}). Recently, we addressed this problem in its full generality (see \cite{KDPconj}), which leads to a characterization of continuous $q$-Jacobi, Chebyshev of the first kind and some special cases of the Al-Salam-Chihara polynomials. Our motivation here is to obtain a full characterization of Askey-Wilson polynomials similar to \eqref{very-classical}. We define the averaging operator by
\begin{align*}
\mathcal{S}_q\, p(x(s))= \frac{1}{2}\Big( p(x(s+1/2))+p(x(s-1/2))\Big),\quad x(s)=\mbox{$\frac12$}(q^s +q^{-s})\;.
\end{align*}
For our purpose, instead of \eqref{very-classical}, let us consider the following difference equation
\begin{align}\label{open-problem}
&(az^2+bz+c)\mathcal{D}_q P_n(z)=a_n\mathcal{S}_q P_{n+1}(z)+b_n\mathcal{S}_q P_n(z)+c_n\mathcal{S}_q P_{n-1}(z),~ z=x(s)\;,
\end{align}
with $c_n\neq 0$. Our objective is then to characterize all OPS that satisfy \eqref{open-problem}. The case where $a=b=0$ and $c=1$ was solved in recently in \cite{KCDMJP2021-a}.  Using ideas developed therein we are going to solve \eqref{open-problem} in his general form. The structure of the paper is as follows. Section 2 presents some basic facts of the algebraic theory of OPS together with some useful results. Section \ref{main} contains our main result. In Section \ref{example} we present a finer result for a special case.  

\section{Background and preliminary results}
The algebraic theory of orthogonal polynomials was introduced by P. Maroni (see \cite{M1991}). Let $\mathcal{P}$ be the vector space of all polynomials with complex coefficients
and let $\mathcal{P}^*$ be its algebraic dual. A simple set in $\mathcal{P}$ is a sequence $(P_n)_{n\geq0}$ such that $\mathrm{deg}(P_n)=n$ for each $n$. A simple set $(P_n)_{n\geq0}$ is called an OPS with respect to ${\bf u}\in\mathcal{P}^*$ if 
$$
\langle{\bf u},P_nP_m\rangle=\kappa_n\delta_{n,m}\quad(m=0,1,\ldots;\;\kappa_n\in\mathbb{C}\setminus\{0\}),
$$
where $\langle{\bf u},f\rangle$ is the action of ${\bf u}$ on $f\in\mathcal{P}$. In this case, we say that ${\bf u}$ is  regular. The left multiplication of a functional ${\bf u}$ by a polynomial $\phi$ is defined by
$$
\left\langle \phi {\bf u}, f  \right\rangle =\left\langle {\bf u},\phi f  \right\rangle \quad (f\in \mathcal{P}).
$$
Consequently, if $(P_n)_{n\geq0}$ is a monic OPS with respect to ${\bf u}\in\mathcal{P}^*$, then the corresponding dual basis is explicitly given by 
\begin{align}\label{expression-an}
{\bf a}_n =\left\langle {\bf u} , P_n ^2 \right\rangle ^{-1} P_n{\bf u}.
\end{align}
Any functional ${\bf u} \in \mathcal{P}^*$ (when $\mathcal{P}$ is endowed with an appropriate strict inductive limit topology, see \cite{M1991}) can be written in the sense of the weak topology in $\mathcal{P}^*$ as 
\begin{align*}
{\bf u} = \sum_{n=0} ^{\infty} \left\langle {\bf u}, P_n \right\rangle {\bf a}_n.
\end{align*}
It is known that a monic OPS, $(P_n)_{n\geq 0}$, is characterized by the following three-term recurrence relation (TTRR):
\begin{align}\label{TTRR_relation}
P_{-1} (z)=0, \quad P_{n+1} (z) =(z-B_n)P_n (z)-C_n P_{n-1} (z) \quad (C_n \neq 0),
\end{align}
and, therefore,
\begin{align}\label{TTRR_coefficients}
B_n = \frac{\left\langle {\bf u} , zP_n ^2 \right\rangle}{\left\langle {\bf u} , P_n ^2 \right\rangle},\quad C_{n+1}  = \frac{\left\langle {\bf u} , P_{n+1} ^2 \right\rangle}{\left\langle {\bf u} , P_n ^2 \right\rangle}.
\end{align}

The Askey-Wilson and the averaging operators  induce two elements on $\mathcal{P}^*$, say $\mathbf{D}_q$ and $\mathbf{S}_q$, via the following definition (see \cite{FK-NM2011}): 
\begin{align*}
\langle \mathbf{D}_q{\bf u},f\rangle=-\langle {\bf u},\mathcal{D}_q f\rangle,\quad \langle\mathbf{S}_q{\bf u},f\rangle=\langle {\bf u},\mathcal{S}_q f\rangle.
\end{align*}

Let $f,g\in\mathcal{P}$ and ${\bf u}\in\mathcal{P}^*$. Hereafter we denote $z=x(s)=(q^s+q^{-s})/2$. Then the following properties hold (see e.g. \cite{KDP2021} and references therein):
\begin{align}
\mathcal{D}_q \big(fg\big)&= \big(\mathcal{D}_q f\big)\big(\mathcal{S}_q g\big)+\big(\mathcal{S}_q f\big)\big(\mathcal{D}_q g\big), \label{def-Dx-fg} \\[7pt]
\mathcal{S}_q \big( fg\big)&=\big(\mathcal{D}_q f\big) \big(\mathcal{D}_q g\big)\texttt{U}_2  +\big(\mathcal{S}_q f\big) \big(\mathcal{S}_q g\big), \label{def-Sx-fg} \\[7pt]
f\mathcal{D}_qg&=\mathcal{D}_q\left[ \Big(\mathcal{S}_qf-\frac{\texttt{U}_1}{\alpha}\mathcal{D}_qf \Big)g\right]-\frac{1}{\alpha}\mathcal{S}_q \Big(g\mathcal{D}_q f\Big) , \label{def-fDxg} \\[7pt]
f{\bf D}_q {\bf u}&={\bf D}_q\left(\mathcal{S}_qf~{\bf u}  \right)-{\bf S}_q\left(\mathcal{D}_qf~{\bf u}  \right), \label{def-fD_x-u}\\[7pt]
\alpha \mathbf{D}_q ^n \mathbf{S}_q {\bf u}&= \alpha_{n+1} \mathbf{S}_q \mathbf{D}_q^n {\bf u}
+\gamma_n \texttt{U}_1\mathbf{D}_q^{n+1}{\bf u}, \label{DxnSx-u} 
\end{align}
where $\alpha =(q^{1/2}+q^{-1/2})/2$, $\texttt{U}_1 (z)=(\alpha ^2 -1)z$ and $\texttt{U}_2 (z)=(\alpha ^2 -1)(z^2-1)$. It is known that 
\begin{align}
\mathcal{D}_q z^n =\gamma_n z^{n-1}+u_nz^{n-3}+\cdots,\quad \mathcal{S}_q z^n =\alpha_n z^n+\widehat{u}_nz^{n-2}+\cdots, \label{Dx-xnSx-xn}
\end{align}
where $$\alpha_n= \mbox{$\frac12$}(q^{n/2} +q^{-n/2})\;,\quad \gamma_n = \frac{q^{n/2}-q^{-n/2}}{q^{1/2}-q^{-1/2}}$$ and, $u_n$ and $\widehat{u}_n$ are some complex numbers. We set $\gamma_{-1}:=-1$ and $\alpha_{-1}:=\alpha$.
It is important to notice that if ${\bf u}$ is a linear functional such that ${\bf D}_q(\phi {\bf u})={\bf S}_q(\psi {\bf u})$, then the following relation holds (see \cite[Proposition 2: (3.8)]{FK-NM2011}) 
\begin{align}\label{key-for-second-order-relation}
\left\langle {\bf u}, (\phi \mathcal{D}^2_qf+\psi \mathcal{S}_q\mathcal{D}_qf)g  \right\rangle=\left\langle {\bf u}, (\phi \mathcal{D}^2_qg+\psi \mathcal{S}_q\mathcal{D}_qg)f  \right\rangle, \;f,g\in \mathcal{P}\;.
\end{align} 
 We denote by $P_n ^{[k]}$ $(k=0,1,\ldots)$ the monic polynomial of degree $n$ defined by
\begin{align*}
P_n ^{[k]} (z)=\frac{\mathcal{D}_q ^k P_{n+k} (z)}{ \prod_{j=1} ^k \gamma_{n+j}} =\frac{\gamma_{n} !}{\gamma_{n+k} !} \mathcal{D}_q ^k P_{n+k} (z). 
\end{align*}
Here it is understood that $\mathrm{D}_x ^0 f=f $, empty product equals one, and $\gamma_0 !=1$, $\gamma_{n+1}!=\gamma_1\cdots \gamma_n \gamma_{n+1}$. If $({\bf a}^{[k]} _n)_{n\geq 0}$ is the dual basis associated to the sequence $(P_n ^{[k]})_{n\geq 0}$, we leave it to the reader to verify that
\begin{align}
{\bf D}_q ^k {\bf a}^{[k]} _n=(-1)^k \frac{\gamma_{n+k}!}{\gamma_n ! }{\bf a}_{n+k}\quad (k=0,1,\ldots). \label{basis-Dx-derivatives}
\end{align}

In 2003, M.E.H. Ismail proved the following result (see \cite[Theorem 20.1.3]{I2005}):
\begin{theorem}\label{T}
A second order operator equation of the form
\begin{align}\label{Ismail}
\phi(z)\mathcal{D}^2_q\, Y+\psi(z) \mathcal{S}_q \mathcal{D}_q \, Y+h(z)\, Y=\lambda_n\, Y
\end{align}
has a polynomial solution $Y_n(z)$ of exact degree $n$ for each $n=0,1,\dots$, if and only if $Y_n(z)$ is a multiple of the Askey-Wilson polynomials, or special or limiting cases of them. In all these cases $\phi$, $\psi$, $h$, and $\lambda_n$ reduce to
\begin{align*}
\phi(z)&=-q^{-1/2}(2(1+\sigma_4)z^2-(\sigma_1+\sigma_3)z-1+\sigma_2-\sigma_4),\\[7pt]
\psi(z)&=\frac{2}{1-q} (2(\sigma_4-1)z+\sigma_1-\sigma_3), \quad h(z)=0,\\[7pt]
\lambda_n&=\frac{4 q(1-q^{-n})(1-\sigma_4 q^{n-1})}{(1-q)^2}, 
\end{align*}
or a special or limiting case of it, $\sigma_j$ being the jth elementary symmetric function of the Askey-Wilson parameters. 
\end{theorem}

\section{Main result}\label{main}
We are now in the position to prove our main result. 
\begin{theorem}\label{propo-sol-q-quadratic}
If $(P_n)_{n\geq 0}$ is a monic OPS such that 
\begin{align}
(az^2+bz+c)\mathcal{D}_qP_{n}(z)= a_n\mathcal{S}_qP_{n+1}(z)+b_n\mathcal{S}_q P_{n}(z)+c_n\mathcal{S}_qP_{n-1}(z)\quad (n=0,1,\ldots)\;,\label{equation-case-deg-is-two}
\end{align}
with $c_n\neq 0$ for $n=0,1,\ldots$, where the constant parameters $a$, $b$ and $c$ are chosen such that
\small{
\begin{align}
&(4\alpha^2-1)aC_2C_3 
+\frac{r_3}{2}\Big[(B_0+B_1)^2 +4\alpha^2 (C_1-B_0B_1 +\alpha^2 -1) -2(2\alpha^2-1)C_2  \Big]=0\;,\label{condition-case-deg-is-two-1}
\end{align}
}
whenever $a\neq 0$, and
\begin{align}
aC_2C_3\Big(b_2+2aB_2+\frac{b}{\alpha}\Big) 
-r_3\left(a(B_2+B_1)C_2  +\frac{b}{\alpha}C_2  -\frac{r_2}{2}(B_1-B_0) \right)=0\;, \label{condition-case-deg-is-two-2}
\end{align}
$r_i=c_i+2aC_i$, $i=2,3$, then $(P_n)_{n\geq 0}$ are multiple of Askey-Wilson polynomials, or special or limiting cases of them. Moreover $(P_n)_{n\geq 0}$ satisfy \eqref{Ismail} with
\begin{align}
&\phi(z)=\mathfrak{a}z^2 +\mathfrak{b}z+\mathfrak{c}\;,~\psi(z)=z-B_0\;,~h(z)=0\;,~ \lambda_n=\gamma_n (\mathfrak{a}\gamma_{n-1}+\alpha_{n-1})\;,\label{expression-phi-psi-general-case}
\end{align}
where
\begin{align*}
&\mathfrak{a}=-\frac{aC_3 +(\alpha^2-1)r_3}{\alpha r_3}\;;\\
&\mathfrak{b}=-\frac{1}{2\alpha}\left(\Big(1-2a\frac{C_3}{r_3}\Big)\Big(B_0+B_1\Big)-2\alpha^2B_0  \right) \;;\\
&\mathfrak{c}=-\frac{1}{2\alpha}\left(\Big(1-2a\frac{C_3}{r_3}\Big)\Big(C_1-B_0B_1\Big)+C_1 +B_0 ^2\right)   \;;
\end{align*}
being $B_0$, $B_1$, $C_1$, $C_2$ and $C_3$ coefficients for the TTRR relation \eqref{TTRR_relation} satisfied by $(P_n)_{n\geq 0}$.
\end{theorem}  

\begin{proof}
Let $(P_n)_{n\geq 0}$ be a monic OPS with respect to the functional ${\bf u} \in \mathcal{P}^*$ and satisfying \eqref{equation-case-deg-is-two}. Set $\pi_2(z)=az^2+bz+c$.  Using \eqref{def-fDxg}, we obtain
\begin{align}\label{start-point-01}
\pi_2\mathcal{D}_q P_n= \mathcal{D}_q \left[ \Big(\mathcal{S}_q\pi_2 -\frac{\texttt{U}_1}{\alpha}\mathcal{D}_q\pi_2 \Big)P_n\right] -\frac{1}{\alpha}\mathcal{S}_q\Big( \mathcal{D}_q\pi_2~P_n  \Big)\;.
\end{align}
By direct computations we obtain $\mathcal{D}_q \pi_2(z)=2\alpha az+b$ and $\mathcal{S}_q\pi_2(z)=a(2\alpha^2-1)z^2+\alpha bz +c +a(1-\alpha^2)$. Therefore 
$$
\mathcal{S}_q\pi_2 -\frac{\texttt{U}_1}{\alpha}\mathcal{D}_q\pi_2 = az^2+\frac{b}{\alpha}z +c +a(1-\alpha^2)\;.
$$
Hence from \eqref{start-point-01}, using the TTRR \eqref{TTRR_relation}, we can rewrite \eqref{equation-case-deg-is-two} as the following
\begin{align}\label{equation-bis}
\sum_{j=n-1} ^{n+1} a_{n,j}\mathcal{S}_qP_j(z)=\sum_{j=n-3} ^{n+1} b_{n,j}P_{j} ^{[1]}(z) \quad \quad (n=0,1,\ldots)\;;
\end{align}
where
\begin{align*}
&a_{n,n+1}=2a+a_n,~a_{n,n}=b_n+2aB_n+\frac{b}{\alpha},~a_{n,n-1}=c_n+2aC_n,\\
&b_{n,n+1}=a\gamma_{n+2}, ~b_{n,n}=\gamma_{n+1}\Big(a(B_{n+1}+B_n)+\frac{b}{\alpha}  \Big),\\
&b_{n,n-2}=\gamma_{n-1}C_n\Big(a(B_{n-1}+B_n)+\frac{b}{\alpha}  \Big),~b_{n,n-3}=a\gamma_{n-2}C_nC_{n-1}\\
&b_{n,n-1}=\gamma_n\left(a(C_{n+1}+B_n ^2+C_n )+\frac{b}{\alpha}B_n +c+a(1-\alpha^2)   \right)\;.
\end{align*}
Let us write $P_n(z)=z^n +f_nz^{n-1}+\ldots$, for $n=0,1,\ldots$, with $f_n=-\sum_{j=0} ^{n-1}B_j$.  Then using \eqref{Dx-xnSx-xn}, we identify the coefficients of the terms in $z^{n+1}$ and in $z^n$ in \eqref{equation-case-deg-is-two} to obtain
\begin{align}
a_n=a\frac{\gamma_n}{\alpha_{n+1}}\;,~
\alpha_n\alpha_{n+1} b_n=\gamma_n(a\alpha_nB_n +b\alpha_{n+1})+a\alpha\sum_{j=0} ^{n-1}B_j \;,
\end{align}
for $n=0,1,\ldots$. Assume now that $a\neq 0$ (the case where $a=0$ follows the same idea as below) and define $$Q_n(z)=\sum_{j=n-2} ^n a_{n-1,j}P_j(z)~~\quad(n=0,1,\ldots)\;.$$ Then $(Q_n)_{n\geq 0}$ is a simple set of polynomials and so let $({\bf a}_n)_{n\geq 0}$, $({\bf a}_n ^{[1]})_{n\geq 0}$ and $({\bf r}_n)_{n\geq 0}$ be the associated basis to the sequences $(P_n)_{n\geq 0}$, $(P_n ^{[1]})_{n\geq 0}$ and $(Q_n)_{n\geq 0}$, respectively. 
We then claim that 
\begin{align}
&{\bf a}_n=a_{n-1,n}{\bf r}_n+a_{n,n}{\bf r}_{n+1}+a_{n+1,n}{\bf r}_{n+2}\;;\label{basis-relation-an-with-rn-1}\\
&{\bf \mathrm{S}}_q {\bf a}_n ^{[1]}=b_{n-1,n}{\bf r}_n+b_{n,n}{\bf r}_{n+1}+b_{n+1,n}{\bf r}_{n+2}+b_{n+2,n}{\bf r}_{n+3}+b_{n+3,n}{\bf r}_{n+4}\;,\label{basis-relation-with-rn-2}
\end{align}
for $n=0,1,\ldots$.

Indeed, we have
$$
\left\langle {\bf a}_n,Q_l  \right\rangle=\sum_{j=l-2} ^l a_{l-1,j}\left\langle {\bf a}_n,P_j  \right\rangle=\sum_{j=l-2} ^l a_{l-1,j}\delta_{n,j}=\left\{
    \begin{array}{ll}
        a_{l-1,n} & \mbox{if } n\leq l\leq n+2 \\
        0 & \mbox{otherwise.}
    \end{array}
\right.
$$
Similarly using \eqref{equation-bis}, we obtain
\begin{align*}
\left\langle {\bf S}_q {\bf a}_n ^{[1]},Q_l  \right\rangle =\left\langle  {\bf a}_n ^{[1]},\mathcal{S}_q Q_l  \right\rangle &=\sum_{j=l-4} ^l b_{l-1,j}\left\langle {\bf a}_n ^{[1]},P_j ^{[1]}  \right\rangle\\
&=\sum_{j=l-4} ^l b_{l-1,j}\delta_{n,j} =\left\{
    \begin{array}{ll}
        b_{l-1,n} & \mbox{if } n\leq l\leq n+4 \\
        0 & \mbox{otherwise.}
    \end{array}
\right.
\end{align*}
Hence \eqref{basis-relation-an-with-rn-1}--\eqref{basis-relation-with-rn-2} follow by writing
$$
{\bf a}_n =\sum_{l=0} ^{+\infty} \left\langle {\bf a}_n ,Q_l  \right\rangle {\bf r}_l ,~\quad {\bf S}_q {\bf a}_n ^{[1]} =\sum_{l=0} ^{+\infty} \left\langle {\bf S}_q {\bf a}_n ^{[1]} ,Q_l  \right\rangle {\bf r}_l \quad \quad (n=0,1,\ldots)\;,
$$
and taking into account what is preceding. Taking $n=0,1,2$ in \eqref{basis-relation-an-with-rn-1} and $n=0$ in \eqref{basis-relation-with-rn-2}, we obtain the following system.
\begin{align}
{\bf a}_0&=a{\bf r}_0+a_{0,0}{\bf r}_{1}+a_{1,0}{\bf r}_{2}  \;;\label{basis-eqt-1}\\
{\bf a}_1&=2a{\bf r}_1+a_{1,1}{\bf r}_{2}+a_{2,1}{\bf r}_{3} \;;\label{basis-eqt-2}\\
{\bf a}_2&=\frac{4\alpha^2-1}{2\alpha^2-1}a{\bf r}_2+a_{2,2}{\bf r}_{3}+a_{3,2}{\bf r}_{4}  \;;\label{basis-eqt-3}\\
{\bf S}_q {\bf a}_0 ^{[1]}&=a{\bf r}_0+b_{0,0}{\bf r}_{1}+b_{1,0}{\bf r}_{2}+b_{2,0}{\bf r}_{3}+b_{3,0}{\bf r}_{4}   \label{basis-eqt-4}     \;.
\end{align}
By subtracting \eqref{basis-eqt-4} to \eqref{basis-eqt-1}, we obtain
\begin{align}
-{\bf a}_0+{\bf S}_q {\bf a}_0 ^{[1]}&=a(B_1-B_0){\bf r}_1+(b_{1,0}-a_{1,0}){\bf r}_{2}+b_{2,0}{\bf r}_{3}+b_{3,0}{\bf r}_{4}   \label{basis-eqt-5}     \;.
\end{align}
We now combine this with \eqref{basis-eqt-2}, we obtain
\begin{align}
-\frac{1}{2}(B_1-B_0){\bf a}_1-{\bf a}_0+{\bf S}_q {\bf a}_0 ^{[1]}&=A{\bf r}_{2}+\Big(b_{2,0}-\frac{1}{2}(B_1-B_0)a_{2,1}\Big){\bf r}_{3}+b_{3,0}{\bf r}_{4}   \label{basis-eqt-6}     \;;
\end{align}
where $A:=b_{1,0}-a_{1,0}-\small{\frac12}(B_1-B_0)a_{1,1}$. Taking $n=1$ in \eqref{equation-case-deg-is-two} we obtain
\begin{align*}
&a_1 =\frac{a}{2\alpha^2-1}\;,~~b_1=\frac{a(B_0+B_1)}{2\alpha^2-1}+\frac{b}{\alpha}\;,~~c_1=\frac{a(B_0 ^2 +C_1 +\alpha^2-1)}{2\alpha^2-1}+c+\frac{b}{\alpha}B_0\;.
\end{align*}
With this we have
\begin{align*}
A=&-\frac{a}{2(2\alpha^2-1)}\Big((B_1+B_0)^2 -4\alpha^2(B_0B_1-C_1+1-\alpha^2) \Big)+aC_2\;.
\end{align*}
Note that $A\neq 0$ according to \eqref{condition-case-deg-is-two-1}. Now, combining \eqref{basis-eqt-6} with \eqref{basis-eqt-3} yields
\begin{align*}
&-\frac{(2\alpha^2-1)A}{(4\alpha^2 -1)a}{\bf a}_2 -\frac{1}{2}(B_1-B_0){\bf a}_1-{\bf a}_0 +{\bf S}_q {\bf a}_0 ^{[1]}\\
&=\Big(b_{2,0}-\mbox{$\frac{1}{2}$}(B_1-B_0)a_{2,1}-\frac{(2\alpha^2-1)A}{(4\alpha^2-1)a}a_{2,2} \Big){\bf r}_3+\Big(b_{3,0}-\frac{(2\alpha^2-1)A}{(4\alpha^2-1)a}a_{3,2}\Big){\bf r}_4\;.
\end{align*}
Using expressions of the coefficients $a_{n,j}$ and $b_{n,j}$ given by \eqref{equation-bis}, we obtain 
$$b_{3,0}-\frac{(2\alpha^2-1)A}{(4\alpha^2-1)a}a_{3,2}=aC_2C_3-\frac{(2\alpha^2-1)A}{(4\alpha^2-1)a}(c_3+2aC_3)=0\;,$$
according to condition \eqref{condition-case-deg-is-two-1}. Similarly, using this relation we also obtain 
\begin{align*}
b_{2,0}-\mbox{$\frac{1}{2}$}(B_1-B_0)a_{2,1}-\frac{(2\alpha^2-1)A}{(4\alpha^2-1)a}a_{2,2}=\Big(a(B_1+B_2)+\frac{b}{\alpha}\Big)C_2\\
-\mbox{$\frac{1}{2}$}(B_1-B_0)(c_2+2aC_2)-\frac{aC_2C_3}{c_3+2aC_3}\Big(b_2+2aB_2+\frac{b}{\alpha}\Big) =0\;,
\end{align*} 
according to condition \eqref{condition-case-deg-is-two-2}. Hence
\begin{align}\label{basis-final}
{\bf S}_q {\bf a}_0 ^{[1]}={\bf a}_0 +\frac{1}{2}(B_1-B_0){\bf a}_1 +\frac{(2\alpha^2-1)A}{(4\alpha^2 -1)a}{\bf a}_2\;.
\end{align}
On the other hand, by direct computations we have $\mathcal{D}_q\texttt{U}_1=\alpha^2-1$ and $\mathcal{S}_q\texttt{U}_1=\alpha\texttt{U}_1$, and therefore, using \eqref{def-fD_x-u} with $f$ and ${\bf u}$ replaced by $\texttt{U}_1$ and ${\bf a}_1$, respectively, we obtain
\begin{align}
\texttt{U}_1{\bf D}_q{\bf a}_1=\alpha {\bf D}_q(\texttt{U}_1{\bf a}_1)-(\alpha^2-1){\bf S}_q{\bf a}_1\;.\label{U1D-x-a1}
\end{align}
We now apply ${\bf D}_q$ to \eqref{basis-final} using successively \eqref{DxnSx-u} for $n=1$ with ${\bf u}$ replaced by ${\bf a}_0$, \eqref{basis-Dx-derivatives} for $k=1$ and $n=0$, and \eqref{U1D-x-a1} to have
\begin{align*}
{\bf D}_q \left[{\bf a}_0 +\frac{1}{2}(B_1-B_0){\bf a}_1 +\frac{(2\alpha^2-1)A}{(4\alpha^2 -1)a}{\bf a}_2  \right]&={\bf D}_q{\bf S}_q{\bf a}_0 ^{[1]}=\frac{\alpha_2}{\alpha}{\bf S}_q{\bf D}_q{\bf a}_0 ^{[1]} +\frac{\texttt{U}_1}{\alpha}{\bf D}_q ^2 {\bf a}_0 ^{[1]} \\
&=-\frac{2\alpha^2-1}{\alpha}{\bf S}_q{\bf a}_1  -\frac{\texttt{U}_1}{\alpha}{\bf D}_q  {\bf a}_1 \\
&=-\alpha {\bf S}_q{\bf a}_1 -{\bf D}_q\big(\texttt{U}_1{\bf a}_1\big) \;.
\end{align*}
Hence
\begin{align*}
{\bf D}_q \left[{\bf a}_0 +\frac{1}{2}\Big((B_1-B_0)+2\texttt{U}_1\Big){\bf a}_1 +\frac{(2\alpha^2 -1)A}{(4\alpha^2 -1)a}{\bf a}_2  \right]= -\alpha {\bf S}_q {\bf a}_1\;.
\end{align*}
So using \eqref{expression-an} and \eqref{TTRR_coefficients} we obtain
$${\bf D}_q(\phi {\bf u})={\bf S}_q(\psi {\bf u})\;,$$
where $\phi$ and $\psi$ are given in \eqref{expression-phi-psi-general-case}. Since ${\bf u}$ is regular, then from \cite[Theorem 4.1]{KDP2021} we obtain $\mathfrak{a}\gamma_n+\alpha_n\neq 0$, $n=0,1,\ldots$. From \eqref{key-for-second-order-relation} with $f=P_n$ and $g=P_l$ we obtain
\begin{align*}
\phi \mathcal{D}^2_qP_n+\psi \mathcal{S}_q\mathcal{D}_qP_n=\sum_{l=0} ^n a_{n,l}P_l\;,
\end{align*}
where 
$\left\langle {\bf u}, P_l ^2\right\rangle a_{n,l} =\left\langle {\bf u}, (\phi \mathcal{D}^2_qP_n+\psi \mathcal{S}_q\mathcal{D}_qP_n)P_l  \right\rangle
=\left\langle {\bf u}, (\phi \mathcal{D}^2_qP_l+\psi \mathcal{S}_q\mathcal{D}_qP_l)P_n  \right\rangle $. Taking into account 
$\phi \mathcal{D}^2_qP_n+\psi \mathcal{S}_q\mathcal{D}_qP_n=\gamma_n(\mathfrak{a}\gamma_{n-1}+\alpha_{n-1})z^n+\textit{lower term degrees}$, we obtain $a_{n,l}=0$ if $l<n$ and $a_{n,n}=\gamma_n(\mathfrak{a}\gamma_{n-1}+\alpha_{n-1})\neq 0$. Therefore
$$\phi(z) \mathcal{D}_q ^2 P_n(z) + \psi(z) \mathcal{S}_q\mathcal{D}_qP_n(z) =a_{n,n}P_n(z)~~\quad(n=1,2,\ldots)\;,$$ 
and the desired result follows by Theorem \ref{T}.

\end{proof}

\section{A special case}\label{example}
In this section we consider the special case of \eqref{equation-case-deg-is-two} where $a=0$ in other to state a finer result. For this purpose we use the following result.
\begin{theorem}\label{main-Thm1}\cite{KDP2021}
Let $(P_n)_{n\geq 0}$ be a monic OPS with respect to ${\bf u} \in \mathcal{P}^*$. 
Suppose that ${\bf u}$ satisfies the distributional equation
$${\bf D}_q(\phi {\bf u})={\bf S}_q(\psi {\bf u})\;,$$
where $\phi(z)=az^2+bz+c$ and $\psi(z)=dz+e$, with $d\neq0$.
Then $(P_n)_{n\geq 0}$ satisfies \eqref{TTRR_relation} with
\begin{align}
B_n  = \frac{\gamma_n e_{n-1}}{d_{2n-2}}
-\frac{\gamma_{n+1}e_n}{d_{2n}},\quad
C_{n+1}  =-\frac{\gamma_{n+1}d_{n-1}}{d_{2n-1}d_{2n+1}}\phi^{[n]}\left( -\frac{e_{n}}{d_{2n}}\right),\label{Bn-Cn-Dx}
\end{align}
 where $d_n=a\gamma_n+d\alpha_n$, $e_n=b\gamma_n+e\alpha_n$, and \begin{align*}
\phi^{[n]}(z)=\big(d(\alpha^2-1)\gamma_{2n}+a\alpha_{2n}\big)
\big(z^2-1/2\big)+\big(b\alpha_n+e(\alpha^2-1)\gamma_n\big)z+ c+a/2,
\end{align*}
\end{theorem}
The monic Askey-Wilson polynomial, $(Q_n(\cdot; a_1, a_2, a_3, a_4 | q))_{n\geq 0}$, satisfy \eqref{TTRR_relation} (see \cite[(14.1.5)]{KLS2010}) with
\begin{align*}
B_n &= a_1+\frac{1}{a_1}-\frac{(1-a_1a_2q^n)(1-a_1a_3q^n)(1-a_1a_4q^n)(1-a_1a_2a_3a_4q^{n-1})}{a_1(1-a_1a_2a_3a_4q^{2n-1})(1-a_1a_2a_3a_4q^{2n})}\\[7pt]
&\quad-\frac{a_1(1-q^n)(1-a_2a_3q^{n-1})(1-a_2a_4q^{n-1})(1-a_3a_4q^{n-1})}{(1-a_1a_2a_3a_4q^{2n-1})(1-a_1a_2a_3a_4q^{2n-2})},\\[7pt]
C_{n+1}&=(1-q^{n+1})(1-a_1a_2a_3a_4q^{n-1}) \\[7pt]
&\quad\times \frac{(1-a_1a_2q^n)(1-a_1a_3q^n)(1-a_1a_4q^n)(1-a_2a_3q^n)(1-a_2a_4q^n)(1-a_3a_4q^n)}{4(1-a_1a_2a_3a_4q^{2n-1})(1-a_1a_2a_3a_4q^{2n})^2 (1-a_1a_2a_3a_4q^{2n+1})}
\end{align*}
and subject to the following restrictions (see \cite{KDP2021}):
$$
\begin{array}l
(1-a_1a_2a_3a_4q^n)(1-a_1a_2q^n)(1-a_1a_3q^n) \\[7pt]
\qquad\quad\times(1-a_1a_4q^n)(1-a_2a_3q^n)(1-a_2a_4q^n)(1-a_3a_4q^n) \neq 0.
\end{array}
$$

We now state our result. 
\begin{corollary} \label{main-thm-q-quadratic-2}

Let $(P_n)_{n\geq 0}$ be a monic OPS satisfying
\begin{align}
(z-r)\mathcal{D}_qP_{n}(z)= b_n\mathcal{S}_qP_n(z)+c_n\mathcal{S}_q P_{n-1}(z)\quad (n=0,1,\ldots)\;,\label{equation-case-deg-is-one-main}
\end{align}
with $c_n\neq 0$ for $n=0,1,\ldots$, where the constant parameter $r$ is chosen such that 
\begin{align}
2(2\alpha ^2 -1)\left(C_2+\alpha (B_1-B_0)r  \right)=B_1^2-B_0^2\;.\label{main-condition-case-deg-is-one-q}
\end{align}
Then $P_n$ is a specific case of the monic Askey-Wilson polynomial:
$$
 P_n(z)=Q_n \Big(z;a_1,a_2,a_3,a_4\Big|q \Big)\quad (n=0,1,\ldots) \;,
$$
where $a_1$, $a_2$, $a_3$ and $a_4$ are complex numbers solutions of the following equation
\begin{align}\label{equation-degree-4-q-a}
Z^4 -RZ^3 +TZ^2-SZ-q^{-1}=0\;,
\end{align} 
with
\begin{align}
\left( R,T,S\right) \in  \left\lbrace \Big(2\frac{qB_0-B_1}{q-1},T_1,2\frac{q^{-1}B_0-B_1}{q-1} \Big),   \right. \label{equation-4-q-case1}\\
\left. \Big(k\frac{(q+1)(2q^2+q+1)}{q^{3/2}(q-1)}, 4\frac{2C_1 + 4\alpha^4 +\alpha ^2 -1}{q-1}, k\frac{(q+1)(q^2+q+2)}{q^{3/2}(q-1)}  \Big),\right.\label{equation-4-q-case2}\\
\left. \Big(2\frac{(q+1)B_0}{q-1},\frac{4(3\alpha^2-1)}{q-1}, 2\frac{(1+q^{-1})B_0}{q-1}  \Big) \right\rbrace\;, \label{equation-4-q-case3}
\end{align}
being $k=\pm 1$ and 
\begin{align*}
T_1=& 1-q^{-1} +8\frac{(B_0^2 -\alpha^2)((4\alpha^2-3)B_1-B_0)}{(B_1+(4\alpha^2-1)B_0)(q-1)} -4\frac{(B_1-B_0)B_0}{q-1}\;.
\end{align*}
\end{corollary}

\begin{proof} 
Let ${\bf u}\in \mathcal{P}^*$ be the regular functional with respect to $(P_n)_{n\geq 0}$. Note that \eqref{main-condition-case-deg-is-one-q} follows from \eqref{condition-case-deg-is-two-2}. Assume that under the condition \eqref{main-condition-case-deg-is-one-q}, $(P_n)_{n\geq 0}$ satisfies \eqref{equation-case-deg-is-one-main}. Then from Theorem \ref{propo-sol-q-quadratic} and his proof,  ${\bf u}$ satisfies the distributional equation ${\bf D}_q(\phi {\bf u})={\bf S}_q (\psi {\bf u})$, where
\begin{align}
&\phi(z)=-\frac{\alpha ^2 -1}{\alpha}z^2 -\frac{1}{2\alpha}\left(B_1 -(2\alpha^2-1)B_0  \right)z+\frac{1}{2\alpha}\left((B_1-B_0)B_0-2C_1 \right) \;,\label{phi-degree-1-q-quadratic}\\
&\psi(z)=z-B_0\;. \label{psi-degree-1-q-quadratic}
\end{align}
We then apply \eqref{Bn-Cn-Dx} to obtain
\begin{align}
&C_2=\frac{1}{4\alpha^2-2}\left[(B_0+B_1)^2-4\alpha^2\left(B_0B_1+1-\alpha^2-C_1   \right)               \right]\;; \label{C-2-q-quadratic}\\
&B_2=-B_0+\frac{2}{4\alpha^2-3}B_1\;.\label{B-2-q-quadratic}
\end{align}
In addition, we claim that the parameters $B_0$, $B_1$ and $C_1$ are related by the following equation
\begin{align}
\left(B_1+(4\alpha^2-1)B_0 \right) C_1=\left( \alpha^2-B_0^2 \right)\big(B_0-(4\alpha^2-3)B_1\big)\;,\label{C1-in-term-of-B0-B1-q}
\end{align}
with $B_0+B_1=0$ or $B_0+B_1 \neq 0$.

Indeed writing $P_n(z)=z^n +f_nz^{n-1}+\cdots$, where $f_0=0$ and $f_n=-\sum_{j=0} ^{n-1}B_j$ for $n=1,2,\ldots$, we identify the coefficients of the two firsts terms with higher degree in \eqref{equation-case-deg-is-one-main} using \eqref{Dx-xnSx-xn} to obtain 
\begin{align}
b_n=\frac{\gamma_n}{\alpha_n},\quad c_n=-r\frac{\gamma_n}{\alpha_{n-1}}+\frac{1}{\alpha_n\alpha_{n-1}} \sum_{j=0} ^{n-1}B_j ~\quad (n=0,1,\ldots)\;.\label{second-hand-datta-case-deg-1-q-quadratic}
\end{align}

Also by direct computations we obtain
\begin{align*}
\mathcal{D}_qP_2(z)=&2\alpha z-B_0-B_1\;,\\
\mathcal{S}_qP_2(z)=&(2\alpha^2-1) z^2 -\alpha(B_0+B_1)z+B_0B_1+1-\alpha^2-C_1\;.
\end{align*}
Similarly we obtain $\mathcal{D}_qP_3$ and $\mathcal{S}_qP_3$ by taking $n=3$ in \eqref{TTRR_relation} using \eqref{def-Dx-fg}--\eqref{def-Sx-fg}:
\begin{align*}
\mathcal{D}_qP_3(z)=&(4\alpha^2-1)z^2-2\alpha(B_0+B_1+B_2)z+B_0B_1+B_0B_2+B_1B_2\\
&-C_1-C_2 +1-\alpha^2 \;,
\end{align*}
and
\begin{small}
\begin{align*}
\mathcal{S}_qP_3(z)&=\alpha(4\alpha^2-3) z^3-(2\alpha^2-1)(B_0+B_1+B_2)z^2\\
&+\alpha \Big(B_0B_1+B_0B_2+B_1B_2-C_2-C_1+3(1-\alpha^2) \Big)z\\
&+B_0C_2+B_2C_1+(\alpha^2-1)(B_0+B_1+B_2)-B_0B_1B_2  \;.
\end{align*}
\end{small}
Taking $n=3$ in \eqref{equation-case-deg-is-one-main}, using what is preceding we obtain the following equations:
\begin{align*}
C_2+C_1+4\alpha^2(\alpha^2-1) +\frac{\alpha(4\alpha^2-3)}{2}\Big(2B_2r+   (B_0+B_1)(c_3+2r)\Big)\\
-B_0B_1-B_0B_2-B_1B_2=0\;, \\
\end{align*}
\begin{align*}
(B_0b_3-r )C_2+(B_2b_3-r-c_3 )C_1+&(\alpha^2-1)b_2(B_0+B_1+B_2+\alpha r)\\
&+(r+c_3)B_0B_1+rB_2(B_0+B_1)=b_3B_0B_1B_2\;,
\end{align*}
where $b_2$ and $b_3$ are giving using \eqref{second-hand-datta-case-deg-1-q-quadratic}.
Hence \eqref{C1-in-term-of-B0-B1-q} is obtained from the previous equations by using the expressions of $c_3$, $B_2$, $C_2$ and $r$ giving by \eqref{second-hand-datta-case-deg-1-q-quadratic}, \eqref{B-2-q-quadratic}, \eqref{C-2-q-quadratic} and \eqref{main-condition-case-deg-is-one-q}, respectively.   
\begin{itemize}
\item[I-]Assume that $B_1=-B_0$.\\
Taking into account that from \eqref{main-condition-case-deg-is-one-q}, $B_1\neq B_0$, we see that in the present case $B_0\neq 0$. From \eqref{C1-in-term-of-B0-B1-q} we have $C_1=\alpha^2-B_0^2 \;,$ then we obtain $\psi(z)=z-B_0$ and $\phi(z)=-(\alpha -\alpha ^{-1})z^2+\alpha B_0 z-\alpha$. This implies that $B_0$ is the only free parameter. Let $a_1$, $a_2$, $a_3$ and $a_4$ be four complex numbers solution of \eqref{equation-degree-4-q-a} for $(R,T,S)$ given by \eqref{equation-4-q-case3}. Then $a_1a_2a_3a_4=-1/q$, $a_1a_2+ a_1a_3 +a_1a_4+a_2a_3+a_2a_4+a_3a_4=4(3\alpha^2-1)/(q-1)$ and 
\begin{align*}
B_0&=\frac{q-1}{2(q+1)}(a_1+a_2+a_3+a_4)\\
&=\frac{q-1}{2(1+q^{-1})}(a_1a_2a_3+a_1a_2a_4+a_1a_3a_4+a_2a_3a_4)\;.
\end{align*}
Hence the result follows by applying \eqref{Bn-Cn-Dx}. Note that $r$ and $c_n$ can be computed using \eqref{main-condition-case-deg-is-one-q} and \eqref{second-hand-datta-case-deg-1-q-quadratic}, respectively. \\

\item[II-] Case where $B_1\neq -B_0$.
\begin{itemize}
\item[II-a] If $B_1=(1-4\alpha^2)B_0$, then since by \eqref{main-condition-case-deg-is-one-q}, $B_1\neq B_0$, we obtain from \eqref{C1-in-term-of-B0-B1-q}, $B_0=\pm \alpha$ and so $B_1=\pm \alpha (1-4\alpha^2)$. Therefore $C_1$ is the only free parameter. In addition we obtain $\psi(z)=z\pm \alpha$ and $\phi(z)=-(\alpha -\alpha ^{-1})z^2 \pm (3\alpha^2-1)z-2\alpha^3  -C_1/\alpha$. Let $a_1$, $a_2$, $a_3$ and $a_4$ be four complex numbers solution of \eqref{equation-degree-4-q-a} for $(R,T,S)$ given by \eqref{equation-4-q-case2}. Then $a_1a_2a_3a_4=-1/q$, $a_1a_2+ a_1a_3 +a_1a_4+a_2a_3+a_2a_4+a_3a_4=4(2C_1 +4\alpha^4+\alpha^2-1)/(q-1)$ , $a_1+a_2+a_3+a_4=kq^{-3/2}(q+1)(2q^2+q+1)/(q-1)$, $a_1a_2a_3+a_1a_2a_4+a_1a_3a_4+a_2a_3a_4=kq^{-3/2}(q+1)(q^2+q+2)/(q-1)$, $k=\pm 1$, and so we obtain
\begin{align*}
C_1&=\frac{q-1}{8}\Big(a_1a_2+a_1a_3+a_1a_4+a_2a_3+a_2a_4+a_3a_4 -\frac{4(4\alpha^4 +\alpha^2-1)}{q-1} \Big) \;.
\end{align*}
Hence the result follows by applying \eqref{Bn-Cn-Dx}. Note that $r$ and $c_n$ can be computed using \eqref{main-condition-case-deg-is-one-q} and \eqref{second-hand-datta-case-deg-1-q-quadratic}, respectively.

\item[II-b] If $B_1\neq (1-4\alpha^2)B_0$, then from \eqref{C1-in-term-of-B0-B1-q}, we obtain can express $C_1$ in term of $B_0$ and $B_1$ and so, for this case $B_0$ and $B_1$ are the only free parameters. Consider four complex numbers $a_1$, $a_2$, $a_3$ and $a_4$ solutions of equation \eqref{equation-degree-4-q-a} with $(R,T,S)$ given by \eqref{equation-4-q-case1}. Proceeding as in the previous cases,
we can write $B_0$, $B_1$ and $C_1$ only in terms of $a_1$, $a_2$, $a_3$ and $a_4$. Therefore from \eqref{phi-degree-1-q-quadratic}--\eqref{psi-degree-1-q-quadratic}, we use \eqref{Bn-Cn-Dx} to obtain the desired result.

\end{itemize}
\end{itemize}
It is important to notice that for each of the above mentioned cases, $a_i$, $i=1,2,3,4$, are roots of the four degree polynomial \eqref{equation-degree-4-q-a} with coefficients depending on some free parameters and so, explicit expressions of these roots can be only given by a computer system. Indeed they always appear in complicated and very large forms and this is why we find unnecessary to write them here but they exist.
\end{proof}
\begin{remark}
This section highlights the use of initial conditions in the considered equation together with Theorem \ref{main-Thm1} to reduce the free parameters and therefore identify the specific polynomial solutions to our problem. Similar investigations can be done for the case where $a\neq 0$ in \eqref{equation-case-deg-is-two}.
\end{remark}

\section{Conclusion}
The results obtained here were proved for the $q$-quadratic lattice
$x(s)=(q^{-s} +q^{s})/2$, where $q$ is not a root of the unity, but it can be easily extended to the quadratic lattice $x(s)=\mathfrak{c}_4s^2+\mathfrak{c}_5s+\mathfrak{c}_6$ by taking the appropriate limit as it was discussed in \cite{KDP2021}. Therefore, by choosing the corresponding quadratic lattice Theorem \ref{propo-sol-q-quadratic} can be easily adapted to characterize the Racah and Wilson polynomials in case there exists the analogue of the Ismail Theorem \ref{T}.

\section*{Declaration}
The authors declare that they have no conflict of interest.

\section*{Acknowledgements }
We would like to thank Kenier Castillo for drawing our attention to the question solved in this work. This work is supported by the Centre for Mathematics of the University of Coimbra-UID/MAT/00324/2019, funded by the Portuguese Government through FCT/MEC and co-funded by the European Regional Development Fund through the Partnership Agreement PT2020. A. Suzuki is also supported by the FCT grant 2021.05089.BD. D. Mbouna thanks the support of the ERDF and Consejeria de Economia, Conocimiento, Empresas y Universidad de la Junta de Andalucia(grant UAL18-FQM-B025-A).

{

\end{document}